\documentclass[12pt]{amsart}%

\usepackage{amsmath, amssymb, amsthm, amsfonts, bbm, dsfont, tabu,stackrel}
\usepackage{graphicx}
\usepackage{float}
\usepackage{wrapfig}
\restylefloat{figure}
\usepackage{mathrsfs}

\usepackage{hyperref}

\usepackage{a4wide}

\numberwithin{equation}{section}
\theoremstyle{plain}
\newtheorem{theorem}[subsubsection]{Theorem}
 \newtheorem{lemma}[subsubsection]{Lemma}
 \newtheorem{proposition}[subsubsection]{Proposition}
 \newtheorem{corollary}[subsubsection]{Corollary}
 \newtheorem{conjecture}[subsubsection]{Conjecture}

 \theoremstyle{definition}

\usepackage{stackengine}
\usepackage{mathtools}
\usepackage{amssymb}
\usepackage{dsfont}

\usepackage{mathabx}

\usepackage{tabu}

\newcommand{\ol}{\overline}

\newcommand{\rmT}{\mathrm{T}}

\newcommand{\CC}{\mathbb{C}}

\newcommand{\LL}{\mathbb{L}}

\newcommand{\ZZ}{\mathbb{Z}}

\newcommand{\bbll}{\mathbb{L}}

\newcommand{\scZ}{\mathscr{Z}}
\newcommand{\scC}{\mathscr{C}}
\newcommand{\scX}{\mathscr{X}}
\newcommand{\scR}{\mathscr{R}}
\newcommand{\scD}{\mathscr{D}}

\newcommand{\bfa}{\mathbf{a}}

\newcommand{\calF}{\mathcal{F}}

\newcommand{\calG}{\mathcal{G}}

\newcommand{\calO}{\mathcal{O}}
\newcommand{\calR}{\mathcal{R}}
\newcommand{\calS}{\mathcal{S}}
\newcommand{\calB}{\mathcal{B}}

\newcommand{\frg}{\mathfrak{g}}

\newcommand{\frb}{\mathfrak{b}}

\newcommand{\frn}{\mathfrak{n}}

\newcommand\aff{\textup{aff}}

\newcommand{\Ind}{\textup{Ind}}
\newcommand{\ind}{\textup{ind}}

\newcommand\Lie{\textup{Lie}\ }
\newcommand\loc{\textup{loc}}

\newcommand\Mod{\textup{Mod}}

\newcommand\per{\textup{per}}

\newcommand\st{\textup{st}}

\newcommand{\Tr}{\textup{Tr}}

\newcommand\Hom{\textup{Hom}}

\newcommand{\HHH}{\mathrm{HHH}}
\newcommand{\rmHH}{\mathrm{HH}}

\newcommand\GL{\textup{GL}}
\newcommand\gl{\mathfrak{gl}}

\newcommand{\Ad}{\textup{Ad}}

\newcommand{\quash}[1]{}

\newcommand{\kn}{\mathrm{kn}}

\newcommand{\Br}{\mathfrak{Br}}

\usepackage{tikz}
\usetikzlibrary{shapes,fit,intersections}
\usetikzlibrary{decorations.markings}
\usetikzlibrary{decorations.pathreplacing}
\usetikzlibrary{patterns}
\usetikzlibrary{matrix,arrows}
\usepgflibrary{decorations.pathmorphing}

\usepackage{tikz-cd}

\newcommand{\Wr}{\overline{W}}
\newcommand{\Hilb}{\textup{Hilb}}

\newcommand{\MF}{\mathrm{MF}}
\newcommand{\calXr}{\overline{\mathcal{X}}}
\newcommand{\calX}{\mathcal{X}}
\newcommand{\MFs}{\mathrm{MF}}
\newcommand{\calZ}{\mathcal{Z}}
\newcommand{\calC}{\mathcal{C}}
\newcommand{\calY}{\mathcal{Y}}

\newcommand{\Fl}{\mathrm{Fl}}

\newcommand{\scz}{\mathscr{Z}}
\newcommand{\frh}{\mathfrak{h}}

\newcommand{\CE}{\mathrm{CE}}

\newcommand{\HH}{\textup{H}}

\newcommand{\odel}{\stackbin[]{\Delta}{\otimes}}

\newcommand{\forg}{\mathrm{fgt}}

\newcommand{\CH}{\mathsf{CH}}
\newcommand{\Dr}{\mathrm{Dr}}
\newcommand{\HC}{\mathrm{HC}}

\newcommand{\rfr}{\mathrm{fr}}

\def\Hilb{ \mathrm{Hilb}}

\newcommand{\ti}{\times}

\newcommand{\ot}{\otimes}

\def\rmD{ \mathrm{D}}

\author{A. Oblomkov}
\address{
A.~Oblomkov\\
Department of Mathematics and Statistics\\
University of Massachusetts at Amherst\\
Lederle Graduate Research Tower\\
710 N. Pleasant Street\\
Amherst, MA 01003 USA
}
\email{oblomkov@math.umass.edu}

\author{L. Rozansky}
\address{
L.~Rozansky\\
Department of Mathematics\\
University of North Carolina at Chapel Hill\\
CB \# 3250, Phillips Hall\\
Chapel Hill, NC 27599 USA
}
\email{rozansky@math.unc.edu}

\title{Matrix factorizations and \(\gl(m|k)\)-quantum invariants}
\begin{document}
\begin{abstract}
In our previous papers we used the Hilbert scheme of points on $\CC^2$ in order to construct a triply graded link homology and its $\gl(m)$ version. Here we extend the $\gl(m)$ construction to super-algebras $\gl(m|k)$.
\end{abstract}

\maketitle
\section{Introduction}
\label{sec:inroduction}
\def\OJ{Ocneanu-Jones}

In our previous papers \cite{OblomkovRozansky16} we developed a matrix factorization approach to the categorification of the \OJ\ trace. The construction relies on the homomorphism
\[\Phi_n: \Br_n\to \MF_n^{\st}=\MF_{\GL_n}\bigl((\gl_n\times \mathrm{T}^*\Fl\times \mathrm{T}^*\Fl\times\CC^n)^{\st},W\bigr),\]
\[W(X,z_1,z_2,v)=\Tr\bigl(X(\mu(z_1)-\mu(z_2))\bigr),\]
where $\MF_{\GL_n}(\ldots,W)$ is the category of $\GL_n$-equivariant matrix factorizations of $W$,  \(\mu:\rmT^*\Fl\to \gl_n\) is the moment map, and the stable locus $(\ldots)^{\st}$ consists of the points \((X,z_1,z_2,v)\) such that
\(\CC\langle X,\mu(z_i)\rangle\, v=\CC^n\). The scaling of \(\gl_n\) and of the cotangent fibers of $\rmT^* \Fl$ factors yields an additional \(T_{qt}=\CC^*\times \CC^*\)-equivariant structure on the
matrix factorizations.

In the subsequent paper \cite{OblomkovRozansky18a} we defined the Chern character \(\CH\) and its right adjoint \(\HC\):
\begin{equation*}
\label{eq:mnchcchf}
  \begin{tikzcd}
    \MF_n^{\st}\arrow[rr,bend left,"\CH^{\st}_{\loc}"]&&\mathscr{D}_{T_{qt}}^{\per}(\Hilb_n)\arrow[ll,bend left,"\HC^{\st}_{\loc}",pos=0.435]
  \end{tikzcd},
\end{equation*}
where \(\mathscr{D}_{T_{qt}}^{\per}(\Hilb_n)\) is a derived category of two-periodic
\(T_{qt}\)-equivariant complexes on the Hilbert scheme of \(n\) points on $\CC^2$.

In the last diagram \(\CH\) is a trace functor and \(\HC\) is monoidal. In particular, in \cite{OblomkovRozansky18a} we define a categorification of 
the \OJ\ trace by
\begin{equation}
\label{eq:ojtr}
\mathcal{T}r(\beta)=\CH(\Phi(\beta))\otimes \Lambda^*\calB,
\end{equation}
where  \(\calB\) is the tautological vector bundle
on the Hilbert scheme, the fiber of the dual vector bundle being defined by
\(\calB^\vee|_I=\CC[x,y]/I\).
We show  in \cite{OblomkovRozansky20} that vector space of the derived global sections
\(\mathbb{H}(\mathcal{T}r(\beta))\) is the triply-graded homology \cite{KhovanovRozansky08a},\cite{KhovanovRozansky08b}, \cite{Khovanov07} \(\HHH(\beta)\) of the closure \(L(\beta)\).

A polynomial \(f\in \CC[x,y]\) determines a Koszul differential \[D_f:\calS\otimes \Lambda^\bullet\calB\to\calS\otimes\Lambda^{\bullet-1}\calB\] for any \(\calS\in \mathrm{D}^{\per}_{T_{qt}}(\Hilb_n)\). Here we show that the choice $f=x^m y^k$ and an insertion of the corresponding differential into the categorified \OJ\ trace~\eqref{eq:ojtr} leads to a categorification of the quantum $\gl(m|k)$ invariants of links.

\begin{theorem}\label{thm:main}
  For any \(\beta\in \Br_r\) the total homology of the complex:
  \[\mathrm{HH}_{m|k}(\beta):=\mathbb{H}(\mathcal{T}r(\beta),D_{x^m y^k})\]
  is a doubly graded vector space that categorifies \(\gl(m|k)\)-quantum invariant of \(L(\beta)\).
\end{theorem}

\def\bfq{ \mathbf{q}}
\def\bft{ \mathbf{t}}
\def\bfa{ \mathbf{a}}

The (triple) $(q,t,a)$-grading structure of the homology \(\HHH(\beta)\) is defined by assigning degrees to the generators $x$ and $y$ of the Hilbert scheme related ring \(\CC[x,y]\) and to the exterior powers of the tautological bundle $\calB$:
\[
\deg x = (1,0,0),\qquad \deg y = (-1,2,0),\qquad \deg \Lambda^i\calB = i (0, 1,1i).
\]

%

The differential \(D_{x^my^k}\)  respects only the double $(Q,T)$ grading, which is defined by similar relations:
\[
\deg x =  (1,0),\qquad \deg y = (-1,2), \qquad \deg\Lambda^i\calB = i(m-k,1,1).
\]
%

Theorem~\ref{thm:main} was already anticipated by many authors. In particular, in \cite{GorskyOblomkovRasmussenShende14} conjectural relation between the
rational DAHAs and triply-graded homology of torus knot \(T_{n,l}\) was proposed.
The Gordon-Stafford theory~\cite{GordonStafford06}  translates the conjecture to the
relation between the sheaves on \(\Hilb_n(\CC^2)\) and triply-graded homology of the torus knots. In this context the differential \(D_{x^m}\) appears naturally and a precise
conjecture for isomorphism  between \(\mathrm{HH}_{m}(T_{n,l})\)  and \(\gl(n)\) homology \cite{KhovanovRozansky08a} is given in \cite{GorskyOblomkovRasmussenShende14}.

In the case of \(m=2,k=0\) the action of the differential \(d_{2|0}\) on \(\lim_{l\to \infty }\HHH(T_{n,l})\) is explored in \cite{GorskyOblomkovRasmussen13} and
\cite{BaiGorskyKivinen19}.

The exploration of the differential similar to
\(D_{x^my^k}\) for general \(m,k\) was initiated in \cite{GorskyGukovStosic18}. Also in \cite{QueffelecRoseSartori18}  a proposal for \(\gl(-m)\) homology is presented. Finally, it is generally anticipated that \(\gl(1|1)\)-homology are related to the Hegaard-Floer homology \cite{Rasmussen15}.

{\bf Acknowledgments} We would like to thank Dmitry Arinkin, Eugene Gorsky, Ivan Losev, Roman Bezrukavnikov and Andrei Negu{\c t} for useful discussions.
The work of A.O. was supported in part by  the NSF CAREER grant DMS-1352398, NSF FRG grant DMS-1760373 and Simons Fellowship.
The work of L.R. was supported in part by  the NSF grant DMS-1108727.

\section{Matrix factorization technology}
\label{sec:matr-fact-techn}

For a proof of the main theorem we need to use some subtle details of the categorification of the braid group from our previous work
\cite{OblomkovRozansky16}. We relied on theory of matrix factorizations and the current section provides a short introduction to the matrix factorization technique of the paper.

  \subsection{Braid group functor}
\label{sec:braid-broup-functor}

In our earlier papers \cite{OblomkovRozansky16}, \cite{OblomkovRozansky19a}, \cite{OblomkovRozansky20} we defined several monoidal functors from the braid
group \(\Br_n\) to categories of matrix factorizations (see the next subsection for more details on matrix factorizations).
 In this note we concentrate on the functor
studied in \cite{OblomkovRozansky16}. We outline the construction of the braid group construction in this subsection and provide more technical details in the later subsections. The key player in the construction is the space with the potential:
\[\calX_n=\frg_n\times G_n\times \frn_n\times G_n\times \frn_n,\quad W(X,g_1,Y_1,g_2,Y_2)=\Tr(X(\Ad_{g_1}Y_1-\Ad_{g_2}Y_2)).\]

Here and everywhere in the paper we use short-hand notations for the groups and Lie algebras:
\[G_n=\GL_n,\quad \frg_n=\gl_n,\quad \frh_n\subset \frb_n\subset \frg_n,\quad \frb_n=\Lie(B_n),\quad \frh_n=\Lie(T_n) \quad \frn_n=[\frb_n,\frb_n]. \]
When the rank of the group is clear from the context the subindex is omitted.

The category of \(G_n\times B_n\times B_n \)-equivariant matrix factorizations has a natural monoidal structure:
\begin{equation}\label{eq:conv-non-red}
\MF_n=\MF_{G_n\times B^2}(\calX_n,W)\ni \calF_1,\calF_2, \quad \calF_1\star\calF_2=\pi_{13*}(\CE_{\frn_n}(\pi_{12}^*(\calF_1)\otimes \pi_{23}^*(\calF_2)))^{T_n},\end{equation}
where \(\CE_\frn(\cdot)\) is a functor of derived \(\frn\)-invariants (see section~\ref{sec:chev-eilenb-compl} and discussion in the papers \cite{OblomkovRozansky16},\cite{OblomkovRozansky19a},\cite{OblomkovRozansky20}) and
the maps in the definition of the convolution \(\star\) are the projections \(\pi_{ij}:\calX^3_n=\frg_n\times (G_n\times \frb_n)^3\to \calX_n\).

\begin{proposition}\label{prop:braids}\cite{OblomkovRozansky16},\cite{OblomkovRozansky17}
  For any \(n\) there is a monoidal functor from the affine braid group \(\Br_n^{\aff}\) to the monoidal category of matrix factorizations:
  \[\Phi_n^{\aff}:\Br_n^{\aff}\to (\MF_n,\star).\]
\end{proposition}

The finite braid group \(\Br_n\)  is naturally a subgroup of \(\Br_n^{\aff}\). In particular, we can restrict the monoidal functor  \(\Phi_n^{\aff}\) to the finite braid group to obtain monoidal functor \(\Phi_n: \Br_n\to \MF_n\)
The inclusion  of the braid groups is a section of the projection
homomorphism \(\forg:\Br_n^{\aff}\to \Br_n\). By introducing stability constraints for matrix factorizations \cite{OblomkovRozansky17} we categorify
the projection homomorphism.

The stable sub-variety of \(\calX_n\times \CC^n\) is defined by the open condition:
\[\left(\calX_n\times \CC^n\right)^{\st}=\{(X,g_1,Y_1,g_2,Y_2,v)|\CC\langle X, \Ad_{g_i}\rangle v=\CC^n,i=1,2\}.\]
The natural projection \(\pi_{v}: \left(\calX_n\times \CC^n\right)^{\st}\to \calX_n \) allows us to pull back the potential
\(W \) to the stable space \(\left(\calX_n\times \CC^n\right)^{\st}\) and use \(W\) for the pull-back \(\pi_{v}^*(W)\).
Thus we define the stable category \cite{OblomkovRozansky16} as:
\[\MF_n^{\st}=\MF_{G_n\times B^2_n}(\left(\calX_n\times \CC^n\right)^{\st},W).\]

The monoidal structure on the category \(\MF_n^{\st}\) and functor \(\Phi_n\) are defined similarly to the case of \(\MF_n\). Moreover, the monoidal structure on \(\MF_n^{\st}\) is compatible with the monoidal structure of \(\MF_n\) in the following sense. An open inclusion \[j_{\st}: \left(\calX_n\times \CC^n\right)^{\st}\to \calX_n\times \CC^n\] induces functor the pull-back functor \(j_{\st}^*\). In \cite{OblomkovRozansky17} we show that diagram of morphisms:
\[\begin{tikzcd}
    \Br_n^{\aff}\arrow[r,"\Phi_n^{\aff}"]\arrow[d,"\forg"]&\MF_n\arrow[d,"j_{\st}^*\circ\pi_v^*"]\\
    \Br_n\arrow[r,"\Phi_n"]&\MF_n^{\st}
  \end{tikzcd}
\]
is a commuting diagram of monoidal functors.

\subsection{Matrix Factorizations}
\label{sec:matr-fact}
Matrix factorizations were introduced by  Eisenbud \cite{Eisenbud80}  and later
the subject was further developed by Orlov \cite{Orlov04}. Below we present only the basic definitions and do not present any proofs.

Let us remind that for an affine variety $\calZ$ and a function $F\in \CC[\calZ]$ there exists a triangulated category $\MFs(\calZ,F)$. The objects of the category are pairs
\[ \mathcal{F}=(M_0\oplus M_1 ,D),\quad D: M_i\rightarrow M_{i+1},\quad D^2=F,\]
where \(M_i\) are free \(\CC[\calZ]\)-modules of finite rank and \(D\) is a homomorphism of  \(\CC[\calZ]\)-modules.

Given \(\mathcal{F}=(M,D)\) and \(\mathcal{G}=(N,D')\) the linear space of morphisms \(\Hom(\mathcal{F},\mathcal{G})\) consists of  homomorphisms of
\(\CC[\calZ]\)-modules
\(\phi=\phi_0\oplus \phi_1\), \(\phi_i\in\Hom(M_i,N_i) \) such that \(\phi\circ D=D'\circ \phi\).
Two morphisms \(\phi,\rho\in \Hom(\mathcal{F},\mathcal{G})\) are homotopic if there is homomorphism of \(\CC[\calZ]\)-modules \(h=h_0\oplus h_1 \),
\(h_i\in \Hom(M_i,N_{i+1})\) such that \(\phi-\rho=D'\circ h-h\circ D\).

In the paper \cite{OblomkovRozansky16} we introduced a notion of  equivariant matrix factorizations which we explain below.
First let us remind the construction of the Chevalley-Eilenberg complex.

\subsection{Chevalley-Eilenberg complex}
\label{sec:chev-eilenb-compl}

Suppose that $\frh$ is a Lie algebra. Chevalley-Eilenberg complex
 $\CE_\frh$ is the complex $(V_\bullet(\frh),d)$ with $V_p(\frh)=U(\frh)\otimes_\CC\Lambda^p \frh$ and differential $d_{ce}=d_1+d_2$ where:
 \def\dtheta{d}
 $$ d_1(u\otimes x_1\wedge\dots \wedge x_p)=\sum_{i=1}^p (-1)^{i+1} ux_i\otimes x_1\wedge\dots \wedge \hat{x}_i\wedge\dots\wedge x_p,$$
 $$ d_2(u\otimes x_1\wedge\dots \wedge x_p)=\sum_{i<j} (-1)^{i+j} u\otimes [x_i,x_j]\wedge x_1\wedge\dots \wedge \hat{x}_i\wedge\dots\wedge \hat{x}_j\wedge\dots \wedge x_p,$$

 Let us denote by $\Delta$ the standard map $\frh\to \frh\otimes \frh$ defined by $x\mapsto x\otimes 1+1\otimes x$.
 Suppose $V$ and $W$ are modules over the Lie algebra $\frh$ then we use notation
 $V\odel W$ for  the $\frh$-module which is isomorphic to $V\otimes W$ as a vector space, the $\frh$-module structure being defined by  $\Delta$. Respectively, for a given $\frh$-equivariant matrix factorization $\calF=(M,D)$ we denote by $\CE_{\frh}\odel \calF$
 the $\frh$-equivariant matrix factorization $(CE_\frh\odel\calF, D+d_{ce})$. The $\frh$-equivariant structure on $\CE_{\frh}\odel \calF$ originates from the
 left action of $U(\frh)$ that commutes with right action on $U(\frh)$ used in the construction of $\CE_\frh$.

 A slight modification of the standard fact that $\CE_\frh$ is the resolution of the trivial module implies that \(\CE_\frh\odel M\) is a free resolution of the
$\frh$-module $M$.

\subsection{Equivariant matrix factorizations}
\label{sec:equiv-matr-fact}

Let us assume that there is an action of the Lie algebra \(\frh\) on \(\calZ\) and \(F\) is a \(\frh\)-invariant function.
Then we can construct the following triangulated category \(\MFs_{\frh}(\calZ,W)\).

The objects of the category are  triples:
\[\mathcal{F}=(M,D,\partial),\quad (M,D)\in\MFs(\calZ,W) \]
where $M=M^0\oplus M^1$ and $M^i=\CC[\calZ]\otimes V^i$, $V^i \in \Mod_{\frh}$,
$\partial\in \oplus_{i>j} \Hom_{\CC[\calZ]}(\Lambda^i\frh\otimes M, \Lambda^j\frh\otimes M)$ and $D$ is an odd endomorphism
$D\in \Hom_{\CC[\calZ]}(M,M)$ such that
$$D^2=F,\quad  D_{tot}^2=F,\quad D_{tot}=D+d_{ce}+\partial,$$
where the total differential $D_{tot}$ is an endomorphism of $\CE_\frh\odel M$, that commutes with the $U(\frh)$-action.


Note that we do not impose the equivariance condition on the differential $D$ in our definition of matrix factorizations. On the other hand, if $\calF=(M,D)\in \MFs(\calZ,F)$ is a matrix factorization with
$D$ that commutes with $\frh$-action on $M$ then $(M,D,0)\in \MFs_\frh(\calZ,F)$.

There is a natural forgetful functor $\MFs_\frh(\calZ,F)\to \MFs(\calZ,F)$ that forgets about the correction differentials:
$$\calF=(M,D,\partial)\mapsto \calF^\sharp:=(M,D).$$

Given two $\frh$-equivariant matrix factorizations $\calF=(M,D,\partial)$ and $\tilde{\calF}=(\tilde{M},\tilde{D},\tilde{\partial})$ the space of morphisms $\Hom(\calF,\tilde{\calF})$ consists of
homotopy equivalence classes of elements $\Psi\in \Hom_{\CC[\calZ]}(\CE_\frh\odel M, \CE_\frh\odel \tilde{M})$ such that $\Psi\circ D_{tot}=\tilde{D}_{tot}\circ \Psi$ and $\Psi$ commutes with
$U(\frh)$-action on $\CE_\frh\odel M$. Two maps $\Psi,\Psi'\in \Hom(\calF,\tilde{\calF})$ are homotopy equivalent if
there is \[ h\in  \Hom_{\CC[\calZ]}(\CE_\frh\odel M,\CE_\frh\odel\tilde{M})\] such that $\Psi-\Psi'=\tilde{D}_{tot}\circ h- h\circ D_{tot}$ and $h$ commutes with $U(h)$-action on  $\CE_\frh\odel M$.

 Given two $\frh$-equivariant matrix factorizations $\calF=(M,D,\partial)\in \MFs_\frh(\calZ,F)$ and $\tilde{\calF}=(\tilde{M},\tilde{D},\tilde{\partial})\in \MFs_\frh(\calZ,\tilde{F})$
 we define $\calF\otimes\tilde{\calF}\in \MFs_\frh(\calZ,F+\tilde{F})$ as the equivariant matrix factorization $(M\otimes \tilde{M},D+\tilde{D},\partial+\tilde{\partial})$.

 \subsection{Push forwards, quotient by the group action}
\label{sec:push-forwards}

The technical part of \cite{OblomkovRozansky16} is the construction of push-forwards of equivariant matrix factorizations. Here we state the main
results, the details may be found in section 3 of \cite{OblomkovRozansky16}. We need push forwards along projections and embeddings. We also use  the
functor of taking quotient by group action for our definition of the convolution algebra.

The projection case is more elementary. Suppose \(\calZ=\mathcal{X}\times\mathcal{Y}\), both \(\calZ \) and \(\mathcal{X}\) have \(\frh\)-action and
the projection \(\pi:\mathcal{Z}\rightarrow\mathcal{X}\) is \(\frh\)-equivariant. Then
for any $\frh$ invariant element $w\in\CC[\calX]^\frh$ there is a functor
\(\pi_{*}\colon \MFs_{\frh}(\calZ, \pi^*(w))\rightarrow \MFs_{\frh}(\mathcal{X},w)
\)
which simply forgets the action of $\CC[\calY]$.


We define an embedding-related push-forward in the case when the subvariety $\calZ_0\xhookrightarrow{j}\calZ$
is the common zero of an ideal $I=(f_1,\dots,f_n)$ such that the functions $f_i\in\CC[\calZ]$ form a regular sequence. We assume that the Lie algebra $\frh$ acts on $\calZ$ and $I$ is $\frh$-invariant. Then there exists an $\frh$-equivariant Koszul complex $K(I)=(\Lambda^\bullet \CC^n\otimes \CC[\calZ],d_K)$ over $\CC[\calZ]$ which has non-trivial homology only in degree zero. Then in section~3 of \cite{OblomkovRozansky16} we define the push-forward functor
\[
j_*\colon \MFs_{\frh}(\calZ_0,W|_{\calZ_0})\longrightarrow
\MFs_{\frh}(\calZ,W),
\]
for any $\frh$-invariant element $W\in\CC[\calZ]^\frh$.


Finally, let us discuss the quotient map. The complex \(\CE_\frh\) is a resolution of the trivial \(\frh\)-module by free modules. Thus the correct derived
version of taking \(\frh\)-invariant part of the matrix factorization \(\mathcal{F}=(M,D,\partial)\in\MFs_\frh(\calZ,W)\), \(W\in\CC[\calZ]^\frh\) is
\[\CE_\frh(\mathcal{F}):=(\CE_\frh(M),D+d_{ce}+\partial)\in\MFs(\calZ/H,W),\]
where \(\calZ/H:=\mathrm{Spec}(\CC[\calZ]^\frh )\) and use the general definition of \(\frh\)-module \(V\):
\[\CE_\frh(V):=\Hom_\frh(\CE_\frh,\CE_\frh\odel V).\]

\section{Markov moves}
\label{sec:markov-moves}

To prove our main result we need to show that the doubly-graded vector space is invariant with respect to the Markov moves and that
compute the character of doubly-graded vector space for the unknot. In this section we work compute deal with the Markov moves. Other  words the main result of this
section is a pair of propositions.

\begin{proposition} \label{prop:move1}For any \(\alpha,\beta\in \Br_n\) and \(m,k\)  we have
  \[\mathrm{HH}_{m|k}(\alpha\beta)=\mathrm{HH}_{m|k}(\beta\alpha).\]
  \end{proposition}

  The braid group \(\Br_n\)  has standard generators \(\sigma_i\), \(i=1,\dots,n-1\) and \(\Br_{n+1}=\langle \Br_n,\sigma_n\rangle\):

  \begin{proposition}\label{prop:move2}
    For any \(\beta\in \Br_n\) and any \(m,k\) we have
    \[\mathrm{HH}_{m|k}(\beta\sigma_n^{\epsilon})=\mathrm{HH}_{m|k}(\beta),\quad \epsilon=\pm 1.\]
  \end{proposition}

It is convenient to introduce a super-group specialization functor to encode the properties of the differential  \(d_{m|k}\).

\subsection{Super-group specialization}
\label{sec:super-group-spec}
Let us introduce two functors on category \(\MF_n^{\st}\) and \(\mathscr{D}_{T_{qt}}(\Hilb_n)\):
\[\calR_{m|k}:\quad\MF_n^{\st}\to \MF_n^{\st},\quad\quad \scR_{m|k}:\quad\scD_{T_{qt}}(\Hilb_n)\to \scD_{T_{qt}}(\Hilb_n).\]
The functor is defined is almost identically for both categories. Indeed, let us fix coordinates on space
\(\calX_n\times \CC^n\) as \((X,g_1,Y_1,g_2,Y_2,v)\).  We describe \(\Hilb_n\) as \(\GL_n\) quotient
of the stable locus of \(T^*\gl_n\times \CC^n\) and the coordinates on the space \(T^*\gl_n\times \CC^n\) are
\((X,Y,v)\). The functors then defined as
\[\calR_{m|k}(\calF)=\calF\otimes (\Lambda^\star\CC^n,d_{m|k}), \quad d_{m|k}=\sum_{i=1}(X^m\Ad_{g_2}(Y_2)^kv)_i\frac{\partial}{\partial\theta_i},\]
\[\scR_{m|k}(\calC)=\calC\otimes (\Lambda^\star\CC^n,d_{m|k}), \quad d_{m|k}=\sum_{i=1}(X^mY^kv)_i\frac{\partial}{\partial\theta_i},\]
where \(\theta_i\),  \(i=1,\dots,n\)  are coordinates along \(\Lambda^\star\CC^n\).

The similarity of the functors translates into the intertwining property:
\begin{proposition}
  For any \(m,k,n\) we have
  \[\CH^{\st}_{\loc}\circ \calR_{m|k}=\scR_{m|k}\circ \CH^{\st}_{\loc}.\]
\end{proposition}
\begin{proof}
  Let us recall the construction of the functor \(\CH^{\st}_{\loc}\) from \cite{OblomkovRozansky18a}.
  First we recall the non-localized functor \(\CH^{\rfr}\) for framed categories:
  \[\CH^{\rfr}:\MF^{\rfr}\to \MF^{\rfr}_{\Dr},\]
  where the framed categories are defined as follows.

  Define framed versions of the spaces \(\scX\) and \(\scC\)
  \[\scX^{\rfr}:=\scX\times V^*_G\times V_{B^{(1)}}\times V_{B^{(2)}},\quad \scC^{\rfr}:=\scC\times V_G^*\times V_G,\quad \scC=\frg\times G\times \frg\]
  where the subindexes indicate the equivariant structure of the vector space \(V=\CC^n\)
and  potentials are
\[W^{\rfr}(X,g_1,Y_1,g_2,Y_2,w,v_1,v_2):=W(X,g_1,Y_1,g_2,Y_2)+\Tr(w(g_1v_1-g_2v_2)),\]
\[W^{\rfr}_{\Dr}(X,g,Z,w,v)=W_{\Dr}(X,g,Z)+\Tr(w(v-gv)),\quad W_{\Dr}(X,g,Z)=\Tr(Z(X-\Ad_gX)).\]

Respectively, we define the framed categories  as
\[ \MF^{\rfr}:=\MF_{G\times B^2}^{\mathbb{T}_{q,t}}(\scX^{\rfr},W^{\rfr}),\quad
  \MF^{\rfr}_{\Dr}:=\MF_{G}^{\mathbb{T}_{q,t}}(\calC^{\rfr}).
\]

  There two auxiliary spaces that are needed for the construction of \(\CH^{\rfr}\):
  \[\scz_{\CH}^{0,\rfr}=\frg\times G\times\frg\times G\times\frn\times V^*_G\times V_G,\quad \scZ_{\CH}=\frg\times G\times\frg\times G\times\frb\times V^*_G\times V_B\]
The action of \(G\times B\) on these spaces is
\[(k,b)\cdot (Z,g,X,h,Y,u,v)=
  (\Ad_{k}(Z),\Ad_k(g),\Ad_k(X),khb,\Ad_{b^{-1}}(Y),ku,kv).\]
The spaces \(\scC\) and \(\scX\) are endowed with the standard \(G\times B^2\)-equivariant structure, the action of  \(B^2\) on \(\scC\) is trivial.
The following maps
\[\pi_{\Dr}(Z,g,X,h,Y,w,v)=(Z,g,X,w,v),\] \[f_\Delta(Z,g,X,h,Y,w,v)=(X,gh,Y,h,Y,w,h^{-1}g^{-1}(v),h^{-1}(v)).\]
are  fully equivariant if we restrict the \(B^2\)-equivariant structure on
\(\scX\) to the \(B\)-equivariant structure via the diagonal embedding
\(\Delta:B\rightarrow B^2\). Let us denote  by $j^0$ is the inclusion map
of \(\scZ_{\CH}^{0,\rfr}\) to \(\scZ_{\CH}^{\rfr}\).

The kernel of the Fourier-Mukai transform is the Koszul matrix factorization
\[\mathrm{K}_{\CH}:=[X-\Ad_{g^{-1}}X,\Ad_hY-Z]\in \MF(\scZ_{\CH},\pi_{\Dr}^*(W_{\Dr})-f^*_\Delta(W)).\]
and we define the Chern functor:
\begin{equation}\label{eq:CH}
  \CH^{\rfr}(\calC):=\pi_{\Dr*}(\CE_{\frn}(\mathrm{K}_{\CH}\otimes (j^0_*\circ f^*_\Delta(\calC)))^{T}).
\end{equation}

Here and everywhere below we use notation \((-)^T\) for \(T\)-invariants.
From the formulas we see that
\(j^{0*}\circ\pi^*_{\Dr}(d_{m|k})=f^*_\Delta(d_{m|k})\). Thus we have intertwining relation:
  \[\CH^{\rfr}\circ \calR_{m|k}=\scR_{m|k}\circ \CH^{\rfr}.\]

  The categories the open subspaces \(\scX^{\mathrm{fs}}\subset\scX^{\rfr}\) and \(\scX^{\mathrm{fs}}_{\Dr}\subset\scX^{\rfr}_{\Dr}\) are defined by the natural stability
  conditions. The respective categories of matrix factorizations are
  \(\MF^{\mathrm{fs}}\) and \(\MF_{\Dr}^{\mathrm{fs}}\) and the Chern functor \(\CH^{\mathrm{fs}}\) between the
  categories is defined by the same formula as before. Thus the intertwining relation
  holds:
  \[\CH^{\mathrm{fs}}\circ \calR_{m|k}=\scR_{m|k}\circ \CH^{\mathrm{fs}}.\]

  The Kn\"orrer periodicity allows us to eliminate the quadratic term
  \(W^{\rfr}-W\) and the corresponding vector space factors \(V_G,V_{B^{(1)}}\) in
  the space \(\scX^{\rfr}\).
  Thus we have an equivalence between \(\MF^{\mathrm{fs}}\) and \(\MF^{\st}\) and the functor
  \(\CH^{\st}:\MF^{\st}\to \MF_{\Dr}^{\mathrm{fs}}\).

  The Kn\"orrer
  periodicity does not interacts with the factor \(V_{B^{(2)}}\) in the space \(\scX^{\rfr}\) hence the corresponding equivalence respects the differential \(d_{m|k}\).
  Thus the intertwining relation for the functor \(\CH^{\st}\) holds.

  The functor \(\CH^{\st}_{\loc}\) is obtained from the functor \(\CH^{\st}\) by
  post composing with the pull-back along the open inclusion map
  \(\loc: \frg^3\times V_G\times V_G^*\to \scC^{\rfr} \) and  the Koszul duality along the
  vector space \(\frg\).

  The composition of the Koszul duality functor with the pull-back \(\loc^*\) is
  equivalent to the pull-back along the inclusion
  \(j_e:\frg^2\times V_G\times V^*_G\to \scC^{\rfr}\) which is induced by the inclusion of
  the identity to the group \(G\). Since \(j^*_e(d_{m|k})=d_{m|k}\) the statement of the proposition follows.
  \end{proof}

The functor \(\calR_{m|k}\) has monoidal properties. More precisely, the functor is bi-modular:
\begin{proposition}
  For any \(m,k,n\) and any \(\calF_i\in \MF_n^{\st}\) we have
  \[\calR_{m|k}(\calF_1\star\calF_2)=\calF_1\star\calR_{m|k}(\calF_2)=\calR_{m|k}(\calF_1)\star\calF_2.\]
\end{proposition}
\begin{proof}
  The convolution in \(\MF^{\st}\) is defined by the formula
  \[\calF_1\star\calF_2=\pi_{13*}(\CE_{\frn_n}(\pi_{12}^*(\calF_1)\otimes \pi_{23}^*(\calF_2)))^{T_n},\]
  where \(\pi_{ij}:\frg\times (G\times\frn)^3\times V\to \frg\times (G\times\frn)^2\times V \) are natural projections.

  Since \(\pi_{13}^*(d_{m|k})=\pi_{23}^*(d_{m|k})\), the first equality is immediate.
  To show the second equality we observe that \(\pi_{13}^*(d'_{m|k})=\pi_{12}^*(d'_{m|k})\)
  where
  \[d'_{m|k}=\sum_{i=1}(X^m\Ad_{g_1}(Y_1)^kv)_i\frac{\partial}{\partial\theta_i}.\]

  Thus to complete the proof we need to show that for any \(\calF\in \MF^{\st}\)
  matrix factorizations \(\calF\otimes\mathcal{K}_{m|k}\) and
  \(\calF\otimes\mathcal{K}'_{m|k}\) are homotopic
  here  \(\mathcal{K}_{m|k}=(\Lambda^*\CC^n,d_{m|k})\) and
  \(\mathcal{K}_{m|k}=(\Lambda^*\CC^n,d'_{m|k})\).

  Let  \(D\) be the differential  of the matrix factorization then the partial derivative \(\frac{\partial D}{\partial X}\) is homotopy between \(\Ad_{g_1}Y_1\)
  and \(\Ad_{g_2}Y_2\). Since \(D\) commutes with \(d_{m|k}\) and the matrix factorization with differentials
   \(D+d'_{m|k}\) and \(D+d_{m|k}\) are homotopic.
\end{proof}

\begin{corollary}\label{cor:trace}
  For any \(\calF_1,\calF_2\in \MF^{\st}_n\) we have
  \[\scR_{m|k}(\CH_{\loc}^{\st}(\calF_1\star\calF_2))= \scR_{m|k}(\CH_{\loc}^{\st}(\calF_2\star\calF_1))\]
\end{corollary}
\begin{proof}
  \begin{multline*}
    \scR_{m|k}(\CH_{\loc}^{\st}(\calF_1\star\calF_2))=\CH_{\loc}^{\st}(\calR_{m|k}(\calF_1\star\calF_2))=\CH_{\loc}^{\st}(\calR_{m|k}(\calF_1)\star\calF_2)=\\
    \CH_{\loc}^{\st}(\calF_2\star\calR_{m|k}(\calF_1))= \CH_{\loc}^{\st}(\calR_{m|k}(\calF_2\star\calF_1))=\scR_{m|k}(\CH_{\loc}^{\st}(\calF_2\star\calF_1)).
  \end{multline*}
\end{proof}

The functor \(\scR_{m|k}\)  provides another presentation for the doubly graded space from the introduction
\[\mathrm{HH}_{m|k}(\beta)=\mathbb{H}(\scR_{m|k}(\CH_{\loc}^{\st}\circ\Phi_n(\beta))).\]
Thus proposition~\ref{prop:move1} follows immediately from the corollary~\ref{cor:trace}. The rest of the section is dedicated to the proof of proposition~\ref{prop:move2}.

We need to recall the construct of the homology from our previous work. In particular, we need to recall the construction for the closure, induction and inclusion of braids.

\subsection{Reduced convolution} To explain induction and inclusion functors we  need to introduce
a smaller {\it `reduced'}
space \(\calXr^\ell_n:=\frb_n\times G^{\ell-1}_n\times\frn_n\) with the \(B^\ell_n\)-action:
\[(b_1,\dots,b_\ell)\cdot(X,g_1,\dots,g_{\ell-1},Y)=(\Ad_{b_1}(X),b_1g_1b_2^{-1},b_2g_2b_3^{-1},\dots,\Ad_{b_\ell}(Y)).\]
In particular the space \(\calXr_n=\calXr_n^2\) has the following \(B^2\)-invariant potential:
\[\Wr(X,g,Y)=\Tr(X\Ad_g(Y)).\]

The proposition 5.1 from \cite{OblomkovRozansky16} provides a functor:
\[\Phi:\MF_{B^2_n}(\calXr_n,\Wr)\rightarrow\MF_{B^2_n}(\calX_n,W)\]
which is an embedding of the categories. Without the \(B^2\)-equivariant structure the functor is an ordinary Kn\"orrer functor \cite{Knorrer}, the equivariant version of
the Kn\"orrer functor is defined as composition of the equivariant pull-back and push-forward (see section 5 of \cite{OblomkovRozansky16}):
\[\Phi_{\kn}:=j^x_*\circ \pi^*_y,\]
where \(\pi_y:\widetilde{\calX}\rightarrow\calXr_n\), \(\widetilde{\calX}:=\frb_n\times G_n\times\frn_n\times G_n\times \frn_n\) is the projection \(\pi_y(X,g_1,Y_1,g_2,Y_2)=
(X,g_1^{-1}g_2,Y_2)\) and \(j^x\) is the natural embedding of \(\widetilde{\calX}\) into \(\calX_n\).

Let us also introduce a convolution algebra structure on the category of matrix factorizations \(\MF_{B^2_n}(\calXr_n,\Wr)\).
There are the following
maps $\bar{\pi}_{ij}:\calXr^3_n\to\calXr_n$:
\[\bar{\pi}_{12}(X,g_{12},g_{13},Y)=(X,g_{12},\Ad_{g_{23}}(Y)_{++}),
  \quad\bar{\pi}_{13}(X,g_{12},g_{13},Y)=(X,g_{12}g_{23},Y),\]
 \[\bar{\pi}_{23}(X,g_{12},g_{13},Y)=(\Ad_{g_{12}}^{-1}(X)_+,g_{23},Y).\]
Here and everywhere below \(X_+\) and \(X_{++}\) stand for the upper and strictly-upper triangular parts of \(X\).
The map \(\bar{\pi}_{12}\times\bar{\pi}_{23}\) is  \(B^2\)-equivariant  but not \(B^3\)-equivariant. However in section 5.4 of \cite{OblomkovRozansky16} we show that
for any \(\mathcal{F},\mathcal{G}\in\MF_{B^2_n}(\calXr_n,\Wr)\) there is a natural element
\begin{equation}\label{eq:conv-red}
(\bar{\pi}_{12}\otimes_B\bar{\pi}_{23})^*(\mathcal{F}\boxtimes\mathcal{G})\in\MF_{B^3}(\calXr_3,\bar{\pi}_{13}^*(W)),
\end{equation}

such that  we can define the binary operation on \(\MF_{B^2}(\calXr,\Wr)\):
\[\mathcal{F}\bar{\star}\mathcal{G}:=\bar{\pi}_{13*}(\CE_{\frn^{(2)}}((\bar{\pi}_{12}\otimes_B\bar{\pi}_{23})^*(\mathcal{F}\boxtimes\mathcal{G}))^{T^{(2)}})\]
and \(\Phi\) intertwines the convolution structures:
\[\Phi_{\kn}(\mathcal{F})\star\Phi_{\kn}(\mathcal{G})=\Phi_{\kn}(\mathcal{F}\bar{\star}\mathcal{G}).\]

As we mentioned in the introduction, it is natural to consider the framed version of our basic spaces.
The framed version of the non-reduced space is an open subset \(\calX_{\st}^\ell\subset\calX^\ell_n\times V\), \(V=\CC^n\) defined by the stability condition:
\[ \CC \langle \Ad_{g_i}^{-1}(X),Y_i\rangle g^{-1}_i(u)=V,\quad i=1,\dots,\ell-1.    \]
Similarly, we define the framed reduced space \(\left(\calXr_n\times V\right)^{\st}=\calXr_{\st}^2\subset\calXr^2_n\times V\) with the stability condition
\begin{equation}\label{eq:stab}
  \CC\langle X,\Ad_g(Y)\rangle u=V.
  \end{equation}

Let us also define \(\calXr_{\st}^3\) to be the intersection \(\bar{\pi}_{12}^{-1}(\calXr_{\st}^2)\cap \bar{\pi}_{23}^{-1}(\calXr_{\st}^2) \) where
\(\bar{\pi}_{ij}\) are the maps \(\calXr^3_n\times V\rightarrow\calXr_n\times V\) which are just extensions of the previously discussed maps by the identity
map on \( V\). Similarly we have the natural maps \(\pi_{ij}:\calX_{\st}^3\rightarrow\calX_{\st}^2\) and
both reduced and non-reduced spaces have natural convolution algebra structure   defined by the formulas (\ref{eq:conv-non-red}) and (\ref{eq:conv-red})

We denote by $j_{\st}$ the maps \(\calX_{\st}^2\rightarrow\calX_n\), \(\calXr_{\st}^2\rightarrow\calXr_n\) that forget the framing. Lemma 12.3
of \cite{OblomkovRozansky16} says that the corresponding pull-back morphism is an homomorphism of the convolution algebras:
\[j_{\st}^*(\mathcal{F}\star\mathcal{G})=j_{\st}^*(\mathcal{F})\star j_{\st}^*(\mathcal{G}).\]
Finally, let us mention that we can restrict the Kn\"orrer functor \(\Phi_{\kn}\) on the open set \(\calXr_{\st}^2\) to obtain the functor
\[\Phi_{\kn}: \MF_{B^2_n}(\calXr_{\st}^2,\Wr)\rightarrow\MF_{B^2_n}(\calX_{\st}^2,W).\]
This functor intertwines the convolution algebra structures on the reduced and non-reduced framed spaces.

\subsection{Induction functors}
\label{sec:induction-functors}
The standard parabolic subgroup \(P_k\) has Lie algebra generated by \(\frb\) and \(E_{i+1,i}\), \(i\ne k\).
Let us define space  \(\calXr(P_k):=\frb\times P_k \times \frn\) and let us also use notation
\(\calXr(G_n)\) for \(\calXr_n\). There is a natural embedding \(\bar{i}_k:\calXr(P_k)\rightarrow\calXr_n\) and
a natural projection \(\bar{p}_k:\calXr(P_k)\rightarrow\calXr(G_k)\times\calXr(G_{n-k})\). The embedding \(\bar{i}_k\) satisfies
the conditions for existence of the push-forward and we can define the induction functor:
\[\overline{\ind}_k:=\bar{i}_{k*}\circ \bar{p}_k^*: \MF_{B_k^2}(\calXr_k,\Wr)\times\MF_{B_{n-k}^2}(\calXr_{n-k},\Wr)\rightarrow\MF_{B_n^2}(\calXr_{n},\Wr)\]

Similarly we define  the space  \(\calXr_{\st}(P_k)\subset\frb\times P_k \times \frn\times V\) as an open subset defined by the stability condition (\ref{eq:stab}).
The last space has a natural projection map \(\bar{p}_k:\calXr_{\st}(P_k)\rightarrow\calXr(G_k)\times\calXr_{\st}(G_{n-k})\) and
the embedding \(\bar{i}_k: \calXr_{\st}(P_k)\rightarrow\calXr_{\st}(G_{n})\) and we can define the induction functor:
\[\overline{\ind}_k:=\bar{i}_{k*}\circ \bar{p}_k^*: \MF_{B_k^2}(\calXr(G_k),\Wr)\times\MF_{B_{n-k}^2}(\calXr_{\st}(G_{n-k}),\Wr)\rightarrow\MF_{B_n^2}(\calXr_{\st}(G_{n}),\Wr)\]

It is shown in section 6 (proposition 6.2) of \cite{OblomkovRozansky16} that the functor \(\overline{\ind_k}\) is the homomorphism of the convolution algebras:
\[\overline{\ind}_k(\mathcal{F}_1\boxtimes\mathcal{F}_2)\bar{\star} \overline{\ind}_k(\mathcal{G}_1\boxtimes\mathcal{G}_2)=
  \overline{\ind}_k(\mathcal{F}_1\bar{\star}\mathcal{G}_2\boxtimes\mathcal{F}_2\bar{\star}\mathcal{G}_2).\]
To define the non-reduced version of the induction functors one needs to introduce the space \(\calX^\circ(G_n)=\frg\times G_n\times\frn\times\frn\) which is a
slice to the \(G_n\)-action on the space \(\calX(G_n)\). In particular, the potential \(W\) on this slice becomes:
\[W(X,g,Y_1,Y_2)=\Tr(X(Y_1-\Ad_g(Y_2))).\]
Similarly to the case of the reduced space, one can define the  space \(\calX^\circ(P_k):=\frg\times P_k\times\frn\times\frn\)   and the corresponding
maps \(i_k:\calX^\circ(P_k)\rightarrow\calX^\circ(G_n)\), \(p_k:\calX^\circ(P_k)\rightarrow\calX^\circ(G_k)\times\calX^\circ(G_{n-k})\). Thus we get a version of the induction functor for non-reduced spaces:
\[{\ind}_k:=i_{k*}\circ p_k^*: \MF_{B_k^2}(\calX(G_k),W)\times\MF_{B_{n-k}^2}(\calX(G_{n-k}),W)\rightarrow\MF_{B_n^2}(\calX(G_{n}),W)\]

It is shown in proposition 6.1 of \cite{OblomkovRozansky16} that the Kn\"orrer functor is compatible with the induction functor:
\[\ind_k\circ(\Phi_{\kn}\times\Phi_{\kn})=\Phi_{\kn}\circ\ind_k.\]

\subsection{Generators of the finite braid group action}
\label{sec:gener-braid-group}
 Let us define \(B^2_n\)-equivariant embedding \(i: \calXr(B_n)\rightarrow\calXr_n\), \(\calXr(B_n):=\frb_n\times B_n\times \frn_n\).
The pull-back of \(\Wr\) along the map \(i\) vanishes and the embedding \(i\) satisfies the conditions for existence of the push-forward
\(i_*:\MF_{B^2_n}(\calXr(B_n),0)\rightarrow \MF_{B^2}(\calXr(G_n),\Wr)\). We denote by \(\underline{\CC[\calXr(B_n)]}\in\MF_{B^2}(\calXr(B_n),0)\) the
matrix factorization with zero differential that is homologically non-trivial only in even homological degree. As it is shown in proposition 7.1 of \cite{OblomkovRozansky16} the
push-forward
\[\bar{\mathds{1}}_n:=i_{*}(\underline{\CC[\calXr(B_n)]})\]
is the unit in the convolution algebra. Similarly, \(\mathds{1}_n:=\Phi_{\kn}(\bar{\mathds{1}}_n)\) is also a unit in non-reduced case.

Let us first discuss the case of the braids on two strands. The key to construction of the braid group action in \cite{OblomkovRozansky16} is the following factorization in the case
\(n=2\):
$$\Wr(X,g,Y)=y_{12}(2g_{11}x_{11}+g_{21}x_{12})g_{21}/\det,$$
where \(\det=\det(g)\) and
$$ g=\begin{bmatrix} g_{11}&g_{12}\\ g_{21}& g_{22}\end{bmatrix},\quad X=\begin{bmatrix} x_{11}&x_{12}\\ 0& x_{22}\end{bmatrix},\quad Y=\begin{bmatrix} 0& y_{12}\\0&0\end{bmatrix}$$
Thus we can define the following strongly equivariant Koszul matrix factorization:
\[\bar{\mathcal{C}}_+:=(\CC[\calXr(G_2)]\otimes \Lambda\langle\theta\rangle,D,0,0)\in\MF_{B^2_2}(\calXr(G_2),\Wr),\]
\[  \quad D=\frac{g_{12}y_{12}}{\det}\theta+\left[g_{11}(x_{11}-x_{22})+g_{21}x_{12}\right]\frac{\partial}{\partial\theta},\]
where \(\Lambda\langle\theta\rangle\)   is the exterior algebra with one generator.

This matrix factorization corresponds to the positive elementary braid on two strands. The negative elementary braid is defined with twisting of the \(B^2_2\)-action. We define:
\[\bar{\calC}_-:=\calC_+\langle-\chi_1,\chi_2\rangle.\]

Using the induction functor we can extend the previous definition on the case of the arbitrary number of strands. For that we introduce an insertion
functor:
\[\overline{\Ind}_{k,k+1}:\MF_{B_2^2}(\calXr(G_2),\Wr)\rightarrow\MF_{B_n^2}(\calXr(G_n),\Wr)\]
\[\overline{\Ind}_{k,k+1}(\calF):=\overline{\ind}_{k+1}(\overline{\ind}_{k-1}(\bar{\mathds{1}}_{k-1}\times \calF)\times\bar{\mathds{1}}_{n-k-1}),\]
and similarly we define non-reduced insertion functor \[\Ind_{k,k+1}:\MF_{B_2^2}(\calX(G_2),W)\rightarrow\MF_{B_n^2}(\calX(G_n),W).\]
Thus we define the generators of the braid group as follows:
\[\bar{\calC}_\epsilon^{(k)}:=\overline{\Ind}_{k,k+1}(\bar{\calC}_\epsilon),\quad \calC_\epsilon^{(k)}:=\Ind_{k,k+1}( \calC_\epsilon).\]

For an element \(\beta\in\Br_n\), if we can choose a presentation
\(\beta=\sigma_{i_1}^{\epsilon_1}\dots\sigma_{i_l}^{\epsilon_l} \) then the element \(\bar{\calC}_\beta:=\bar{\calC}_{\epsilon_1}^{(i_1)}\bar{\star}\dots\bar{\star}\bar{\calC}_{\epsilon_l}^{(i_l)}\) does not depends on the presentation of the braid. Respectively, we use notation
\(\calC_\beta=\Phi_{\kn}(\bar{\calC}_\beta)\).  Respectively, the finite part \(\Phi_n\) of the homomorphism \(\Phi_n^{\aff}\) from proposition~\ref{prop:braids} is defined by \(\Phi_n(\beta)=\calC_\beta\).

\subsection{Braid closure}
\label{sec:braid-closure}

The braid closure functor \(\LL\) is very natural in context of the reduced spaces:
\[\LL: \ol{\MF_n^{\st}}\to \scD_{B}(\left(\frb\times\frn\times \CC^n\right)^{\st}), \quad \LL(\bar{\calF})=j_e^*(\bar{\calF}).\]
where \(j_e:
\left(\frb_n\times 1\times\frn_n\times\CC^n\right)^{\st}\to\left(\calXr_n\times\CC^n\right)^{\st} \) is the inclusion map.

As it is shown in \cite{OblomkovRozansky16} the \(B\) acts freely on
\((\frb\times\frn\times \CC^n)^{\st}\) and the corresponding quotient is a smooth manifold that
we call free flag Hilbert scheme:
\[\mathrm{FHilb}^{\mathrm{free}}_n=(\frb\times\frn\times \CC^n)^{\st}/B.\]

Let \((X,Y,v)\) be coordinates on the product \(\frb\times\frn\times \CC^n\) then
similarly, to the previous case we define super-group restriction functor:
\[\scR'_{m|k}(\calC)=\calC\otimes (\Lambda^\star\CC^n,d_{m|k}), \quad d_{m|k}=\sum_{i=1}(X^mY^kv)_i\frac{\partial}{\partial\theta_i}.\]

The closure functor intertwines the functors \(\scR'_{m|k}\) and \(\scR_{m|k}\) because of base change relation:
\[\scR'_{m|k}\circ \LL\circ \Phi=\LL\circ \Phi\circ \scR_{m|k}.\]

It is shown in \cite{OblomkovRozansky18a} that for any \(\calF\in \MF^{\st}_n\) we have:
\[\Hom(\calF,\calC_1)=\mathbb{H}(\CH^{\st}_{\loc}(\calF)).\]
On the other hand in \cite{OblomkovRozansky16} it was shown that
\[\Hom(\Phi(\bar{\calF}),\calC_1)=\CE_{\frn}(\LL(\bar{\calF}))^{T}.\]
Combining all these observations we obtain another construction for the homology from the introduction
\[\mathrm{HH}_{m|k}(\beta)=\CE_{\frn}(\scR'_{m|k}\circ\LL(\bar{\calC}_\beta))^T.\]

To complete our proof of the second Markov move we need to  combine the
above description of the homology together with a technical lemma from
\cite{OblomkovRozansky16} (also see \cite{OblomkovRozansky20} for further discussion of
this lemma). To state the lemma we need to set up some notations.

The  \(B^2_n\)-equivariant projection
\[\pi:\frn_n\ti \frb_n\to \frb_{n-1}\ti \frb_{n-1},\]
that projects matrices to the subspace spanned by the matrix units \(E_{ij}\), \(i,j>1\) descends to the projection
\(\pi\) between the free flag Hilbert scheme:
\(\mathrm{FHilb}_n^{\mathrm{free}}\to\mathrm{FHilb}_{n-1}^{\mathrm{free}}\). On the other hand the projection
\[\pi_{\ge k}:\frn_n\ti\frb_n\to \frn_k\ti \frb_k.\]
that projects matrices to the subspace spanned by the matrix units \(E_{ij}\), \(i,j\le k\) does not descend to
to the map between the free flag Hilbert schemes. The  failure is due to the fact that the stability condition is incompatible
with the projection. However, we can use this projection to understand the homology of the graphs that appear in the
the Markov move theorem:

\begin{lemma}\label{lem:Markov}
  Let \(\overline{\calF}\in \overline{\MF}_{n-1}\), \(\overline{\calG}\in \overline{\MF}_2\) and \(V=\CC^{n-2}\langle \chi_1\rangle \)  is a vector space
  with the  \(B_n\)-action given by the character \(\chi_1=\exp(\epsilon_1)\).
  Suppose we have \[\bbll(\overline{\calF})=(M,d)\in \rmD^{\per}_{B^2}(\frn_{n-1}\ti\frb_{n-1})\]
then
  there is a deformed
  \(\ZZ_2\)-graded complex \[\overline{\calF}'=(\pi^*M\ot \Lambda^{even}V\oplus \pi^*M\ot \Lambda^{odd}V,D_V+\pi^* d)\in \rmD^{\per}_{B}(\frn_n\ti \frb_n)\]
  \begin{equation*}\label{dia:big-1}
 D_V:    \begin{tikzcd}[column sep=small]
             \pi^* M\ar[r]\ar[rrr,bend right=15]\ar[rrrrr,bend right=20]&
      \pi^* M\ot V\ar[r]\ar[rrr,bend right=15]&
       \pi^*M\ot\Lambda^2V\ar[r]\ar[rrr,bend right=15]&
       \pi^*M\ot\Lambda^3V\ar[r]&
      \pi^*M\ot\Lambda^4 V\ar[r]&\cdots
  \end{tikzcd}
\end{equation*}
such that
\[\bbll\bigg(\overline{\ind}_1(\overline{\calF})\star \overline{\mathrm{Ind}}_{1,2}(\overline{\calG})\bigg)=\overline{\calF}'\ot \pi^*_{\ge 2}(\bbll(\overline{\calG})).\]
\end{lemma}


\subsection{Proof of the second Markov move}
\label{sec:proof-second-markov}
To show proposition~\ref{prop:move2} we apply the previous lemma for
\(\bar{\calF}=\bar{\calC}_\beta\) and \(\bar{\calG}=\bar{\calC}_\epsilon\).
First we observe that \(\LL(\bar{\calC}_\epsilon)\) is a Koszul complex
\[
  K(x_{11}-x_{22})\langle \frac{\epsilon-1}{2}\chi_1 \rangle,\quad K(x_{11}-x_{22})=
[R\xrightarrow{x_{11}-x_{22}}R]\]
here and everywhere below \(R=\CC[(\frb\times\frn\times \CC^n)^{\st})]\).

Next, observe for any \(\calC\in \scD_B^{\per}((\frb_{n-1}\times \frn_{n-1}\times \CC^n)^{\st})\) we have
\[\scR'_{m|k}(\pi^*(\calC))=\pi^*(\scR'_{m|k}(\calC))\otimes [R\to R\langle\chi_1\rangle ]\]

The morphism \(\pi\) is \(B_n\)-equivariant and we denote by \(\hat{\pi}\) the
induced morphism on the  quotients.
The \(B_n\)-action on \((\frb\times \frn\times \CC^n)^{\st}\) is free and
it shown in \cite{OblomkovRozansky16} that the fibers of
the projection
\[\hat{\pi}: (\frb_n\times\frn_n\times \CC^n)^{\st}/B_n\to(\frb_{n-1}\times\frn_{n-1}\times \CC^{n-1})^{\st}/B_{n-1}
\]
are products \(\CC\times \mathbb{P}^{n-1}\) where the coordinate along the first
component is \(x_{11}-x_{22}\).

The functor \(\CE_{\frn_n}(\cdot)^{T_n}\) is a composition of
two functors. The first functor is the push-forward along \(\pi_{\ge 1}\times \hat{\pi}\)
and the second functor is \(\CE_{\frn_{n-1}}(\cdot)^{T_{n-1}}\). Thus to complete our
proof we need to combine  the projection formula for \((\pi_{\ge 1}\times \hat{\pi})_*\) with
the computation of the push-forward \((\pi_{\ge 1}\times \pi)_*\) applied to
\[(\Lambda^*V,D_V)\otimes [R\to R\langle\chi_1\rangle] \langle \frac{\epsilon-1}{2}\chi_1 \rangle \]

If \(\epsilon=+\) the last push-forward is \(\calO\) since the only term
of the last tensor product that contribute non-trivially to the push-forward is
\(\Lambda^0V\otimes R\).  If \(\epsilon=-\) the last push-forward is \(\calO[n-1]\) since the only term
of the last tensor product that contribute non-trivially to the push-forward is
\(\Lambda^{n-2}V[n-2]\otimes R[1]\langle\chi_1\rangle\).

\section{Evaluation of unknot and final remarks}
\label{sec:evaluation-unknot}

\subsection{Unknot}
\label{sec:unknot}

To complete proof of the main theorem we need to evaluate the homology of the unknot. For that we observe that
\[(\calXr_1\times \CC)^{\st}=\CC_x\times\CC^*\times \{0\}\times \CC^*_v.\]
Since for \(n=1\) we have \((\frb\times\frn\times V)^{\st}=\CC_x\times\CC^*_v\), \((\frb\times\frn\times V)^{\st}/\CC^*=\CC_x\)
the closure functor results in
\[\LL(\Phi_1(1))=\CC[\CC_x\times \CC^*_v],\quad \mathcal{T}r(1)=\CC[\CC_x]=\CC[x].\]

The study the differential \(D_{x^my^k}\) we split on two cases. The first case \(k=0\) then homology \(\HH_{m}(1)\) is computed by the complex
\[\mathbf{Q}^m\mathbf{T}\cdot\CC[x]\xrightarrow{x^m}\CC[x].\]
In the second case \(k\ge 0\) and the differential vanishes for \(n=1\), that is \(\HH_{m|k}(1)\) is homology of
\[\mathbf{Q}^{m-k}\mathbf{T}\cdot\CC[x]\xrightarrow{0}\CC[x].\]
To sum it up we have
\[\dim_{\mathbf{Q},\mathbf{T}}\rmHH_{m}(1)=(\mathbf{Q}^m-1)/(\mathbf{Q}-1), \quad
\dim_{\mathbf{Q},\mathbf{T}}\rmHH_{m|k}(1)=(\mathbf{T}\mathbf{Q}^{m-k}+1)/(1-\mathbf{Q}).\]

Thus indeed, \(\rmHH_{m|k}\) categorifies \(\gl(m|k)\) invariant since
\[\dim_{\mathbf{Q},\mathbf{T}=-1}\rmHH_{m|k}(1)=(1-\mathbf{Q}^{m-k})/(1-\mathbf{Q}).\]

Let us also remark that in the case \(m=k\) the last expression vanishes and some renormalization is needed to obtain a homology theory that specializes to the Alexander polynomial. We will that discussion for the future work. We also expected the technology of dualizable homology from \cite{OblomkovRozansky19a} provide a path to
proving a \(qt\)-symmetry conjecture.

\subsection{Conjectures}
\label{sec:conjectures}

We have shown that that the homology in the note decategorify to the \(\gl(n|k)\)
link invariant. In the case \(k=0\) a categorification of the knot invariant was constructed by Khovanov \cite{Khovanov00} and Khovanov-Rozansky homology
\(\mathrm{H}_{\gl(n)}^{\mathrm{KhR}}\), \cite{KhovanovRozansky08a}.

It is also shown in \cite{OblomkovRozansky20} that  \(\mathcal{T}r\) provides a categorification of Jones-Ocneanu trace that matches the categorification by Khovanov
\cite{Khovanov07}. Thus we expect that a combination of the results from
\cite{OblomkovRozansky20} and known results on the relations between the Soergel category and the categories related to Khovanov-Rozansky \(\gl(k)\)-invariants would yield a proof of the following

\begin{conjecture} For \(k\ge 2\) we have for any \(\beta\in \Br_n\)  \[\mathrm{H}_{0|k}(\beta)=\mathrm{H}_{\gl(k)}^{\mathrm{KhR}}(L(\beta)).\]
\end{conjecture}

In more details, the category related to the Khovanov-Rozansky \(\gl(k)\)-invariants is a quotient
of the category of singular Soergel bimodules \(\mathrm{SBim}_n^{\mathrm{s}}\) by a subcategory generated by
objects
\[B^{(r_1,r_2)}=R_n\otimes_{R_m^{S([r_1,r_2])}}R_n,\quad r_1-r_2\ge k,\]
where \(R_n=\CC[x_1,\dots,x_n]\),
\(S([r_1,r_2])\subset S_n\) is a the Young subgroup that fixes \(1,\dots,r_1-1\)
and \(r_2+1,\dots,n\). On the other hand in \cite{OblomkovRozansky20} we construct
fully-faithful monoidal functor from some natural subcategory \(\MF^\flat_n\subset \MF_n^{\st}\)
category of matrix factorizations to Soergel bimodules
\[\mathbb{B}: \MF_n^{\flat}\to \mathrm{SBim}_n.\]

This later category can be enlarged to category \(\MF_n^{\flat}\) to slightly
larger category \(\MF_n^{\flat,\mathrm{s}}\subset \MF_n^{\st}\) so that we could
extend the functor \(\mathbb{B}\) to the functor from \(\MF_n^{\flat,\mathrm{s}}\)
to \(\mathrm{SBim}_n^{\mathrm{s}}\). Let us describe the objects that corresponding to
the bimodules \(B^{(k_1,k_2)}\).

First, let us recall that used a category that is equivalent to the category \(\MF_n^{\st}\) in \cite{OblomkovRozansky20}:
\[\MF_n^{\st}=\MF_{G\times B^2}((\frg\times G\times \frb\times G\times \frb\times V_B)^{\st},W),\]
where the potential \(W\) and the stability condition is the same as before, see section~\ref{sec:braid-broup-functor}.

The matrix factorizations \(\calC_\bullet^{(1,r)}\in \MF_r^{\st}\)  are defined as Koszul
matrix factorizations:
\[(R\otimes \Lambda^*(\theta_{ij}),D,0,0), \quad R=\CC[(\frg\times G\times \frb\times G\times \frb\times V_B)^{\st}],\]
\[D=\sum_{ij} X_{ij}\frac{\partial}{\partial X_{ij}}+Z_{ij}\theta, \quad W=
  \sum_{ij}X_{ij}Z_{ij},\]
where we used coordinates \((X,g_1,Y_1,g_2,Y_2,v)\) on \((\frg\times G\times \frb\times G\times \frb\times V_B)^{\st}\).

The objects \(\calC_\bullet^{(r_1,r_2)}\), \(r_2-r_1=r-1\) are obtained from \(\calC_\bullet^{(1,r)}\) by means of the induction functors discussed in the section
\ref{sec:induction-functors}.

The key observation toward a proof of the conjecture is the following
statement for \(\calC_{\bullet}^{(r_1,r_2)}\in \MF^{\flat,\mathrm{s}}_n\):
\[\mathbb{B}\circ \mathcal{R}_{0|k}(\calC_\bullet^{(r_1,r_2)})=0, \quad r_2-r_1\ge k. \]

It is a straight forward to check that the last equality is compatible with the
induction functors. Thus it is enough to check the statement for \(r_1=1,r_2=n\). So let us recall a construction of the functor \(\mathbb{B}\) from \cite{OblomkovRozansky20}:
\[\mathbb{B}(\calC)=\pi_*(j_{x=0}^*(\calC)^{G\times B^2}),\]
where \(j_{x=0}\) is the embedding of \(((G\times \frb)^2\times V_B)^{\st}\)
inside \( (\frg\times(G\times \frb)^2\times V_B)^{\st} \) and \(\pi\) is the
projection from \(((G\times \frb)^2\times V_B)^{\st}\) to \(\frh^2\).

Thus we the pull back along \(j_{x=0}\) results into the Koszul complex:
\[j^*_{x=0}(\calC_\bullet^{(1,n)})=K[Z_{ij}].\]
Since \(Z_{ij}=(\Ad_{g_1}Y_1-\Ad_{g_2}Y_2)\), to complete our computations
of \(\mathbb{B}(\calC_\bullet^{(1,n)})\) we need to \(G\times B^2\) invariant
part of
\[i_{z=0}^*((\Lambda^*V_B,d_{0|k})), \quad i_{z=0}: \{Z_{ij}=0\}\to ((G\times \frb)^2\times V_B)^{\st}. \]

The \(G\times B^2\)-action on the subspace defined by the equations \(Z_{ij}=0\) is free and a slice to the action has the following parametrization:
\[g_1=g_2=1,\quad Y_1=Y_2=\mathrm{diag}(\vec{y})+J,\quad v=e_n,\]
where \((e_n)_i=\delta_{n-i}\) and \(J_{ij}=\delta_{j-i-1}\).

On the slice the differential \(d_{0|k}\) defines the complete intersection variety defined by the equation \(Y_1^kv=0\). Finally, we can observe that
\((\mathrm{diag}(\vec{y})+J)^k_{n-k,n}=1\) hence the equation \(Y_1^kv=0\) cuts out
empty set on the slice.

Let us also remark that in \cite{OblomkovRozansky19a} we prove \(q\to -q/t\) symmetry of HOMFLYPT homology. We expect that the method of the mentioned paper yields a proof of the following

\begin{conjecture}
  Let \(\pi_0(L(\beta))=1\) then
  \[\rmHH_{m|m}(\beta)=\rmHH_{m|m}(\beta)|_{\mathbf{Q}\to \mathbf{T}/\mathbf{Q}}.\]
\end{conjecture}


\begin{thebibliography}{GORS14}

\bibitem[BGK19]{BaiGorskyKivinen19}
Yuzhe Bai, Eugene Gorsky, and Oscar Kivinen.
\newblock Quadratic ideals and rogers–ramanujan recursions.
\newblock {\em The Ramanujan Journal}, 52(1):67–89, May 2019.

\bibitem[Eis80]{Eisenbud80}
D.~Eisenbud.
\newblock Homological algebra on a complete intersection, with an application
  to group representations.
\newblock {\em Trans. Amer. Math. Soc.}, (1):35--64, 1980.

\bibitem[GGS18]{GorskyGukovStosic18}
Eugene Gorsky, Sergei Gukov, and Marko Stošić.
\newblock Quadruply-graded colored homology of knots.
\newblock {\em Fundamenta Mathematicae}, 243(3):209–299, 2018.

\bibitem[GOR13]{GorskyOblomkovRasmussen13}
Eugene Gorsky, Alexei Oblomkov, and Jacob Rasmussen.
\newblock On stable khovanov homology of torus knots.
\newblock {\em Experimental Mathematics}, 22(3):265–281, Jul 2013.

\bibitem[GORS14]{GorskyOblomkovRasmussenShende14}
E.~Gorsky, A.~Oblomkov, J.~Rasmussen, and V.~Shende.
\newblock Torus knots and the rational {DAHA}.
\newblock {\em Duke Mathematical Journal}, 163:2709--2794, 2014.

\bibitem[GS06]{GordonStafford06}
I.~Gordon and J.~T. Stafford.
\newblock {Rational Cherednik algebras and Hilbert schemes, II: Representations
  and sheaves}.
\newblock {\em Duke Mathematical Journal}, 132(1):73--135, Mar 2006.

\bibitem[Kho00]{Khovanov00}
M.~Khovanov.
\newblock A categorification of the {Jones} polynomial.
\newblock {\em Duke Mathematical Journal}, 101(3):359--426, Feb 2000.

\bibitem[Kho07]{Khovanov07}
M.~Khovanov.
\newblock Triply-graded link homology and { Hochschild} homology of {Soergel}
  bimodules.
\newblock {\em International Journal of Mathematics}, 18(08):869--885, Sep
  2007.

\bibitem[Kno87]{Knorrer}
H.~Knorrer.
\newblock {Cohen-Macaulay} modules on hypersurface singularities. {I}.
\newblock {\em Inventiones Mathematecae}, (1):153--164, 1987.

\bibitem[KR08a]{KhovanovRozansky08a}
M.~Khovanov and L.~Rozansky.
\newblock Matrix factorizations and link homology.
\newblock {\em Fundamenta Mathematicae}, 199:1--91, 2008.

\bibitem[KR08b]{KhovanovRozansky08b}
M.~Khovanov and L.~Rozansky.
\newblock Matrix factorizations and link homology {II}.
\newblock {\em Geometry and Topology}, 12:1387--1425, 2008.

\bibitem[OR18a]{OblomkovRozansky17}
A.~Oblomkov and L.~Rozansky.
\newblock Affine braid group, {JM} elements and knot homology.
\newblock {\em Transformation Groups}, Jan 2018.

\bibitem[OR18b]{OblomkovRozansky18a}
A.~Oblomkov and L.~Rozansky.
\newblock {Categorical Chern character and braid groups}, 2018.

\bibitem[OR18c]{OblomkovRozansky16}
A.~Oblomkov and L.~Rozansky.
\newblock Knot homology and sheaves on the {Hilbert} scheme of points on the
  plane.
\newblock {\em Selecta Mathematica}, 24(3):2351--2454, Jan 2018.

\bibitem[OR19]{OblomkovRozansky19a}
A.~Oblomkov and L.~Rozansky.
\newblock {Dualizable link homology}, 2019.

\bibitem[OR20]{OblomkovRozansky20}
A.~Oblomkov and L.~Rozansky.
\newblock {Soergel bimodules and matrix factorizations}, 2020.

\bibitem[Orl04]{Orlov04}
D.~Orlov.
\newblock Triangulated categories of singularities and {D}-branes in
  {Landau-Ginzburg} models.
\newblock {\em Proc. Steklov Inst. Math.}, 246(3):227--248, 2004.

\bibitem[QRS18]{QueffelecRoseSartori18}
Hoel Queffelec, David E.~V. Rose, and Antonio Sartori.
\newblock {Annular Evaluation and Link Homology}, 2018.

\bibitem[Ras15]{Rasmussen15}
J.~Rasmussen.
\newblock Some differentials on {Khovanov–Rozansky} homology.
\newblock {\em Geometry and Topology}, 19(6):3031--3104, Dec 2015.

\end{thebibliography}


\end{document}